\documentclass[oneside,reqno]{amsart} 

\newcounter{version}
\newcounter{short}
\newcounter{long}
\usepackage{amsmath,amsthm,amscd}
\usepackage{amssymb,amsfonts} 

\theoremstyle{definition}
\newtheorem{dfn}{Definition}[section]

\newtheorem{rem}[dfn]{Remark}
\theoremstyle{plain}
\newtheorem{thm}[dfn]{Theorem}
\newtheorem{prp}[dfn]{Proposition}
\newtheorem{cor}[dfn]{Corollary}
\newtheorem{lem}[dfn]{Lemma}

\newcommand\C{{\mathbb C}}
\newcommand\R{{\mathbb R}}
\newcommand\Z{{\mathbb Z}}
\newcommand\N{{\mathbb N}}

\newcommand\A{{\mathbb A}}
\newcommand\Sp{{\mathbb S}}

\newcommand\LL{{\mathcal L}}
\newcommand\CC{{\mathcal C}}
\newcommand\HH{{\mathcal H}}

\newcommand\DD{{{\mathcal D}_X}}
\newcommand\Dh{{\mathcal D}_{\mathcal H}}
\newcommand\DG{{{\mathcal D}_G}}
\newcommand\DS{{{\mathcal D}_{\Sp^{N-1}}}}

\newcommand\m{{\mathfrak m}}

\newcommand\id{{\rm Id}}
\newcommand\Ind{{\rm Ind}}
\newcommand\ad{{\rm ad\,}}
\newcommand\ch{{\rm char\,}}
\newcommand\Hom{{\rm Hom}}
\newcommand\Ker{{\rm Ker\,}}
\newcommand\IM{{\rm Im\,}}
\newcommand\gr{{\rm gr\,}}

\newcommand\Der{{\rm Der\,}}
\newcommand\Jac{{\rm Jac}}
\newcommand\del{\partial}
\newcommand\rank{{\rm rank}}
\newcommand\detab{{\det(\alpha, \beta)}}
\newcommand\sing{{\rm sing}}
\newcommand\Ising{I_\sing}

\newcommand\whp{{\widehat \pi}}
\newcommand\HD{{\widehat{\mathcal D}}}
\newcommand\grD{{{\rm gr}\, \HD}}

\newcommand\dx{{\frac{\partial}{\partial x}}}
\newcommand\dy{{\frac{\partial}{\partial y}}}

\newcommand\me{{\mathfrak m}_e}
\newcommand\Te{\left( \me / \me^2 \right)^*}
\newcommand\Tf{\me / \me^2 }
\newcommand\wv{\widehat{\varphi}}

\newcommand\T{T}
\newcommand\Ha{H}

\newcommand\GCD{gcd}

\date{}
\begin{document}

\

\title
[Lie algebras of vector fields]
{Lie algebras of vector fields on smooth  affine varieties}
\author{Yuly Billig}
\address{School of Mathematics and Statistics, Carleton University, Ottawa, Canada}
\email{billig@math.carleton.ca}
\author{Vyacheslav Futorny}
\address{ Instituto de Matem\'atica e Estat\'\i stica,
Universidade de S\~ao Paulo,  S\~ao Paulo,
 Brasil}
 \email{futorny@ime.usp.br}
\let\thefootnote\relax

\begin{abstract}
We  reprove the results of Jordan \cite{Jo} and Siebert \cite{S} and show that the Lie algebra  of polynomial vector fields on an irreducible affine variety $X$  is simple if and only if $X$ is a smooth variety.  Given proof is self-contained and does not depend on papers mentioned above. Besides, the structure  of the module of  polynomial functions 
 on an irreducible smooth affine variety over the Lie algebra of vector fields is studied. 
Examples of Lie algebras of polynomial vector fields on  an $N$-dimensional sphere, non-singular hyperelliptic curves and linear algebraic groups are considered. 
\end{abstract}

\maketitle

\section{Introduction.}

Classification of complex simple finite dimensional Lie algebras by Killing (1889), \cite{Ki} and Cartan (1894) \cite{C1} shaped the development of Lie theory in the first half of the 20th century.  
Since Sophus Lie the Lie groups and corresponding Lie algebras (as infinitesimal transformations) were related to the symmetries of geometric structures which need not be finite dimensional.  
Later infinite dimensional Lie groups and algebras were connected with the symmetries of systems which have an infinite
number of independent degrees of freedom, for example in Conformal
Field Theory, e.g. \cite{BPZ}, \cite{TUY}.  The discovery of first four classes of simple infinite dimensional Lie algebras goes back to Sophus Lie who introduced certain pseudo groups of transformations in small dimensions. 
This work was completed by Cartan who showed that corresponding simple Lie algebras are of type  
$W_n$, $S_n$, $H_n$ and $K_n$ \cite{C2}. These four classes of Cartan type algebras were the first examples of simple infinite dimensional Lie algebras. The general theory of infinite dimensional Lie algebras at large is still undeveloped. 

  Guillemin and Sternberg \cite{GS},  Singer and Sternberg \cite{SS} and Weisfeiler \cite{W} related Cartan's classification to the study of infinite dimensional $\mathbb Z$-graded Lie algebras of finite depth. 
Generalizing these results Kac \cite{K1} classified all simple complex $\mathbb Z$-graded Lie algebras of finite growth under certain technical restriction.  Besides four Cartan type series, loop Lie algebras  
(whose universal central extensions are affine Kac-Moody algebras) appeared in his classification opening a new theory of Kac-Moody algebras. Finally, Mathieu \cite{M} completed the classification of all 
simple complex infinite dimensional  $\mathbb Z$-graded Lie algebras of finite growth by adding to the list the first Witt algebra (centerless Virasoro algebra) which was conjectured by Kac. 
The latter algebra is 
the Lie algebra of polynomial vector fields on a circle.  Well-known generalized Witt algebras is  
  a natural generalization of Virasoro algebra. These are simple Lie algebras of polynomial vector fields on $n$-dimensional torus. We refer to \cite{BF1}, \cite{BF2} for recent developments in the representation theory of these algebras.

All smooth vector fields $X(M)$  on some manifold $M$ form a natural infinite dimensional Lie algebra. 
 These Lie algebras and their cohomologies have many important properties and applications  \cite{FG}. 
The  Lie algebra $X(M)$  is not simple in general.  The ideal structure of these Lie algebras was studies by many authors \cite{G1}, \cite{G2}, \cite{SP}, \cite{V}.
For a compact manifold $M$, Shanks and Pursell \cite{SP} showed that maximal ideals
of $X(M)$ consist of those vector fields which are
flat at a fixed point of $M$. A larger class of the Lie algebras of vector
fields (vector fields tangent to a given foliation
or analytic) was studied by Grabowski \cite{G3} who showed that the maximal ideals are always modules over the rings of the
corresponding class of functions as was conjectured by Vanzura \cite{V}.


Recall that Cartan type Lie algebras can be constructed as the derivation algebra of the polynomial algebra in several
variables (type $W_n$) and its subalgebras preserving certain differential forms in other types. Similarly,  the generalized Witt algebras are the derivation algebras of the Laurent polynomial algebra. Attempts to construct new generalized Cartan type algebras using this approach go back to the work of Kac \cite{K}.  
New simple Lie algebras of generalized Witt type were constructed by 
Kawamoto \cite{Ka}  
by replacing the polynomial algebra  by the group algebra
of an additive subgroup of a  vector space over the field of characteristic $0$ and considering the Lie subalgebra of its 
derivation algebras generated by the grading operators. Based on these ideas, Osborn \cite {O} constructed four new series of generalized simple Lie algebras of Cartan
type followed by further generalizations in \cite{DZ}, \cite{X1}, \cite{X2}, \cite{J}, \cite{IKN}, \cite{P}.  

A significant breakthrough were two papers of Jordan \cite{Jo}, \cite{Jo1} where the author gave sufficient conditions for the simplicity of the derivation algebra of a commutative, associative  unital $k$-algebra 
$R$  over a field  $k$. If $D$ is both a Lie subalgebra
and $R$-submodule of $Der_k R$ such that $R$ is $D$-simple (that is $R$ has no nontrivial $D$-invariant ideals) and either $char k \neq 2$ or $D$
is not cyclic as an $R$-module or $D(R) = R$, then  $D$ is simple. 
 The result of Jordan was completed  to the criterion  of  simplicity of the Lie algebra of polynomial vector fields on 
an irreducible affine variety by Siebert in \cite{S}. 

In this paper we reprove the results of Jordan and Siebert and give a self-contained proof of the following 

\begin{thm}[\cite{Jo}, \cite{S}]\label{thm-introd}
Let $k$ be an algebraically closed field of characteristic $0$, and let $X \subset \A_k^n$ be an irreducible affine variety. 
 The Lie algebra $\DD$ of polynomial vector fields on $X$ is simple if and only if $X$ is a smooth variety.
\end{thm}

\medskip

 The method of the proof is based on a local-global
principle. First we establish local properties of ideals in the Lie algebra of vector fields.
In particular, we show that every non-zero ideal is locally ample, i.e., vector fields in the ideal, specialized at any non-singular point, span the whole tangent space at that point.
 Next we apply the  developed technique to construct subspaces 
in the ideal, closed under the left action of the algebra of polynomial functions. Finally, we apply Hilbert's Nullstellensatz to globalize the local information and prove that every non-zero ideal must coincide with the whole Lie algebra of vector fields.
 The proof given in the paper is self-contained and does not depend on the results of \cite{Jo} and \cite{S}.

The structure of the paper is the following. In Section 2 we obtain auxiliary results for the Lie algebras of rational vector fields. In Section 3 the main result Theorem \ref{thm-introd} is proved. In Section 4 we consider the structure  of the algebra of polynomial functions $A$ on an irreducible smooth affine variety as a module over the Lie algebra of vector fields. Theorem \ref{thm-functions} shows that $A$ has length $2$. We consider example of the algebra of polynomial functions on an $(N-1)$-dimensional sphere and discuss its structure as a module over $sl_N$ (Proposition \ref{prop-sphere}). In Section 5 the Lie algebra of vector fields on hyperelliptic curves is considered. It is shown that  the Lie algebra of polynomial vector fields on a non-singular hyperelliptic curve has no nontrivial semisimple or nilpotent elements (Theorem \ref{thm-hyper}). Finally, in Section 6 we consider the Lie algebra of polynomial vector fields  $\mathcal D_G$  on a linear algebraic group $G$. 

\medskip

\noindent{ \bf Acknowledgements}. Y.B. is supported in part by NSERC grant and by FAPESP grant (2015/05859-5). He gratefully acknowledges the hospitality and excellent working conditions at
the University of S\~ao Paulo where this work was done. V.F. is supported in part by CNPq grant (301320/2013-6) and by FAPESP 
grant (2014/09310-5).

\medskip

\section{Rational vector fields.}
In this section we are going to show that the Lie algebra of vector fields on an affine variety $X$ with coefficients in rational functions, is simple. We shall prove this result in the following more general setting:
\begin{thm}
\label{field}
Let $k \subset F$ be an extension of fields, $\ch k \neq 2$ and $\Der_k (F)\neq 0$. Let $\LL$ be 
a non-zero Lie subalgebra in $\Der_k (F)$, which is also a left $F$-submodule in $\Der_k (F)$.
Then $\LL$ is a simple Lie algebra over $k$. In particular, $\Der_k (F)$ is a simple Lie algebra.
\end{thm}

Let us start in even more general setting. Let $K \subset S$ be an extension of 
commutative rings. Assume that $2$ is invertible in $K$.
The ring of derivations $\Der_K (S)$ is the space of $K$-linear maps $\eta: \, S \rightarrow S$, satisfying the Leibniz rule: $\eta (fg) = \eta(f) g + f \eta(g)$, for all $f, g \in S$. This space is closed under the Lie bracket and has a natural structure of a left $S$-module. Note that the Lie bracket is not $S$-linear, instead, these two structures are related in the following way:
\begin{equation}
\label{bracket}
[f \eta, g \mu] = fg [\eta, \mu] + f \eta(g) \mu - g \mu(f) \eta,
\hskip 0.6cm
\text{for \ } f, g \in S, \ \eta, \mu \in  \Der_K (S).
\end{equation}
The proof of Theorem \ref{field} will be based on the following four Lemmas.

\begin{lem}
\label{switch}
Let $f, h \in S$, $\eta, \mu \in \Der_K (S)$. Then
\begin{equation*}
(a) \ [\mu, f\eta] - [f\mu, \eta] = \eta(f) \mu + \mu(f) \eta, \hskip 6.2cm
\end{equation*}
\begin{equation*}
(b) \ [\mu, fh \mu] - [f \mu, h \mu] = 2 h \mu(f) \mu.  \hskip 6.5cm
\end{equation*}
\end{lem}
\begin{proof}
This follows immediately from (\ref{bracket}).
\end{proof}

Let $L$ be a Lie subring in $\Der_K (S)$, which is also a left $S$-submodule 
 in $\Der_K (S)$.

\begin{lem}
\label{self}
 Let $J$ be an ideal in $L$, and let $\mu \in J$. Then for all
$f, g \in S$,

(a) $\mu(f) \mu \in J$,

(b) $\mu(f) \mu(g) \mu \in J$,

(c) $f \mu ( \mu(g) ) \mu \in J$.
\end{lem}
\begin{proof}
To establish part (a), consider the commutator $[\mu, f \mu] = \mu(f) \mu \in J$.
Applying  Lemma  \ref{switch} (b) with $h = \mu(g)$, we establish part (b).
Now (c) follows from (a) and (b) since
\begin{equation*}
f \mu ( \mu (g)) = \mu( f \mu(g)) - \mu(f) \mu(g).
\end{equation*} 
\end{proof}

\begin{lem}
\label{grab}
Let $f, g \in S$ and let $\mu\in L$. Let $I_{f,g,\mu}$ be the principal ideal in $S$ generated by $\mu(f) \mu( \mu(g))$, and let $J$ be an ideal in $L$ containing $\mu$.
Then for every $q \in I_{f,g,\mu}$ and every $\tau \in L$, $q\tau \in J$. 
\end{lem}
\begin{proof}
Write $q = p \mu(f) \mu( \mu(g) )$. Then by Lemma \ref{switch} (a),
\begin{equation*}
[p \mu (\mu (g)) \mu, f \tau] - [f p \mu (\mu (g)) \mu, \tau] 
- \tau(f) p \mu ( \mu (g) ) \mu = q \tau .
\end{equation*} 
By Lemma \ref{self} (c), the left hand side belongs to $J$.
\end{proof}

\begin{lem}
\label{nonzero}
Suppose $S$ does not have non-zero nilpotent elements. Then
for every non-zero $\mu \in \Der_K (S)$ there exists $g \in S$ such that 
$\mu ( \mu (g) ) \neq 0$.
\end{lem}
\begin{proof}
Since $\mu \neq 0$, there exists $f \in S$ such that $\mu(f) \neq 0$. 
If $\mu (\mu (f) ) \neq 0$, we can take $g = f$. If $\mu (\mu (f) ) = 0$, 
we take $g = f^2$ and get
\begin{equation*}
\mu (\mu (f^2) ) = 2 \mu ( f \mu (f) ) = 2 \mu(f) \mu(f) + 2 f \mu( \mu (f) )
= 2 \mu (f)^2 \neq 0.
\end{equation*}
\end{proof}

Now we can prove Theorem \ref{field}. We need to show that the ideal $J \lhd \LL$ 
generated by an arbitrary non-zero element $\mu$ is equal to $\LL$.

Using Lemma \ref{nonzero}, choose $f, g \in F$ such that 
$\mu(f) \neq 0$, $\mu (\mu (g) ) \neq 0$. The ideal $I_{f,g,\mu}$ has a non-zero 
generator, and since $F$ is a field, $I_{f,g,\mu} = F$. By Lemma \ref{grab}, 
every $\tau \in \LL$ belongs to $J$.

\section{Polynomial vector fields.}

In this section we prove simplicity of the Lie algebra of polynomial vector fields on 
a smooth irreducible affine variety.  First we establish local properties of ideals in the Lie algebra of vector fields.
In particular, we show that every non-zero ideal is locally ample, i.e., vector fields in the ideal, specialized at any non-singular point, span the whole tangent space at that point.
 The results of the previous section are then applied  to construct subspaces 
in the ideal which are closed under the left action of the algebra of polynomial functions. 
Then application of the Hilbert's Nullstellensatz  globalizes the local information and proves the simplicity.

Let $k$ be an algebraically closed field of characteristic $0$, and let $X \subset \A_k^n$ be an irreducible affine variety.  Let $I$ be the ideal of functions in $k[x_1,\ldots, x_n]$
that vanish on $X$, and let $A = k[x_1,\ldots, x_n]/I$ be the algebra of polynomial functions on $X$. Since $X$ is irreducible, the ideal $I$ is prime, and $A$ has no divisors of zero.

Let us present an explicit vector field realization of the Lie algebra $\Der_k (A)$.
We start with the Lie algebra $W_n$ of vector fields on the affine space $\A_k^n$:
\begin{equation*}
W_n = \mathop\oplus\limits_{j=1}^n k[x_1,\ldots, x_n] \frac{\del}{\del x_j}.
\end{equation*} 
\begin{prp}
\label{threedef}
There exist natural isomorphisms between the following Lie algebras:

(1) $\Der_k (A)$,

(2) The subquotient in $W_n$:
\begin{equation*}
\DD = \left\{ \mu \in W_n \, \big| \, \mu(I) \subset I \right\}\big/
\left\{ \mu \in W_n \, \big| \, \mu( k[x_1,\ldots, x_n] ) \subset I \right\},
\end{equation*}

(3) $\Hom_A ( \Omega_A, A)$, where $\Omega_A$ is the space of K\"ahler differentials of $A$.
\end{prp}

\begin{proof}
We are going to establish the isomorphism between (1) and (2). For the isomorphism between (1) and (3) we refer to \cite{E}, Chapter 16. We are not going to use 
 K\"ahler differentials in this paper, so part (3) is included here only for completeness.

It is clear that every $\mu \in W_n$ satisfying $\mu (I) \subset I$, yields a well-defined 
derivation of $A = k[x_1,\ldots, x_n] / I$. This gives a homomorphism of Lie algebras
\begin{equation*}
\varphi: \ \left\{ \mu \in W_n \, \big| \, \mu(I) \subset I \right\}
\rightarrow \Der_k (A).
\end{equation*}
The subspace $\left\{ \mu \in W_n \, \big| \, \mu( k[x_1,\ldots, x_n] ) \subset I \right\}$
is the kernel of $\varphi$, which gives us an injection
\begin{equation*}
{\overline\varphi}: \ 
\DD \hookrightarrow  \Der_k (A).
\end{equation*}
It remains to show that ${\overline\varphi}$ is surjective. Let $d \in \Der_k (A)$.
Suppose $d(x_j + I) = g_j +I$  for some $g_j \in  k[x_1,\ldots, x_n]$, 
$j=1, \ldots, n$.
Consider the derivation $\eta =\sum\limits_{j=1}^n g_j \frac{\del}{\del x_j} \in W_n$.
Then for $f \in k[x_1,\ldots, x_n]$,
\begin{equation*}
\eta( f(x_1,\ldots, x_n)) + I = d ( f (x_1 + I, \ldots, x_n + I)).
\end{equation*} 
If $f \in I$ then the above equality implies $\eta(f) \in I$. Hence
$\eta \in  \left\{ \mu \in W_n \, \big| \, \mu(I) \subset I \right\}$,
${\overline\varphi} (\eta) = d$, and ${\overline\varphi}$ is an isomorphism of Lie algebras.
\end{proof} 


Consider a set of generators $\{ f_1, \ldots, f_m \}$ of ideal $I$. A vector field
$\eta = \sum\limits_{j=1}^n g_j \frac{\del}{\del x_j}$ belongs to subalgebra
$\left\{ \mu \in W_n \, \big| \, \mu(I) \subset I \right\}$ if and only if for all
$i  = 1, \ldots, m$, 
\break
$\eta(f_i) = \sum\limits_{j=1}^n g_j \frac{\del f_i}{\del x_j} \in I$.

The subspace  $\left\{ \mu \in W_n \, \big| \, \mu( k[x_1,\ldots, x_n] ) \subset I \right\}$
can be simply written as $\mathop\oplus\limits_{j=1}^n I  \frac{\del}{\del x_j}$.
Thus the quotient
\begin{equation*}
\DD = \left\{ \mu \in W_n \, \big| \, \mu(I) \subset I \right\}\big/
\left\{ \mu \in W_n \, \big| \, \mu( k[x_1,\ldots, x_n] ) \subset I \right\}
\end{equation*}
is a left $A$-submodule in a free $A$-module
\begin{equation*}
\mathop\oplus\limits_{j=1}^n k[x_1,\ldots, x_n]  \frac{\del}{\del x_j} \bigg/
\mathop\oplus\limits_{j=1}^n I  \frac{\del}{\del x_j}
\ \cong \  
\mathop\oplus_{j=1}^n A \frac{\del}{\del x_j} .
\end{equation*}
Viewing $\{g_1, \ldots, g_n \}$ as elements of $A$, we obtain the following criterion
for $\eta =  \sum\limits_{j=1}^n g_j \frac{\del}{\del x_j}$ to be a vector field on $X$:
\begin{equation}
\label{system}
\left\{
\begin{matrix}
&\frac{\del f_1}{\del x_1} g_1 &+ & \ldots & +& \frac{\del f_1}{\del x_n} g_n = 0, \cr
&&&\ldots &&\cr
&\frac{\del f_m}{\del x_1} g_1 &+ & \ldots & +& \frac{\del f_m}{\del x_n} g_n = 0. \cr 
\end{matrix}
\right.
\end{equation}
Thus $\DD$ is an $A$-module of solutions in $A^n$ of this system of linear homogeneous equations with coefficients in $A$. One technical difficulty here
is that $\DD$ need not be a free $A$-module (a well-known example of this is $X = \Sp^2$).

The situation simplifies when we solve this system of equations over the field $F$ of
fractions of $A$. The set of solutions of (\ref{system}) is of course a free $F$-module.
Let $r$ be the rank over $F$ of the Jacobian matrix
\begin{equation*}
\Jac = 
\begin{pmatrix}
\frac{\del f_1}{\del x_1}  & \ldots & \frac{\del f_1}{\del x_n}  \cr
&\ldots &\cr
\frac{\del f_m}{\del x_1}  & \ldots & \frac{\del f_m}{\del x_n}  \cr 
\end{pmatrix}.
\end{equation*}
Since the Jacobian matrix is polynomial, it can be 
specialized at any point of $X$. It follows from the Nullstellensatz that
\begin{equation*}
\rank_F \left(\Jac \right) = \max_{P\in X} \, \rank_k \left( \Jac(P) \right) .
\end{equation*}
Then the dimension $s = n - r$ of the solution space over $F$ is equal to the dimension of the variety $X$ (\cite{Sha}, Section II.1.4, Theorem 3). 
Singular points $P$ in $X$ are characterized by the condition (\cite{Sha}, Section II.1.4)
\begin{equation*}
\rank_k \left( \Jac(P) \right) < \rank_F \left(\Jac \right) .
\end{equation*}
For two $r$-element subsets $\alpha \subset [1 \dots n]$ and 
$\beta \subset [1 \dots m]$,
denote by $\detab\in A$ the corresponding $r \times r$ minor in $\Jac$. The locus of singular points in $X$ is given by the zeros in $X$ of the ideal $\Ising \lhd A$ generated 
by all minors $\detab$. 

Assume now that $P = (p_1, \ldots, p_n)$ is a non-singular point of $X$. Without loss
of generality we assume that the principal minor of $\Jac(P)$ (the one formed by the first 
$r$ rows and columns) is non-zero. Then the same minor of $\Jac$ is a non-zero polynomial function $h \in A$, $h(P) \neq 0$.

The maximal ideal $\m_P \lhd A$ of functions on $X$ vanishing at $P$, is generated
by $x_1 - p_1$, \ldots, $x_n - p_n$. Taking the Taylor expansion at $P$ of 
$f_1, \ldots, f_m$ and passing to the quotient modulo $\m_P^2$, we obtain 
linear relations on  $x_1 - p_1$, \ldots, $x_n - p_n$ in $\m_P / \m_P^2$ with matrix
$\Jac(P)$.

Since $\Jac(P)$ has rank $r$ and its principal $r \times r$ minor is non-zero, 
we conclude that
\begin{equation}
\label{param}
t_1 = x_{r+1} - p_{r+1}, \ldots, t_s = x_n - p_n
\end{equation} 
form a basis of $\m_P / \m_P^2$.
Hence these functions may be chosen as local parameters at $P$  (\cite{Sha}, 
Section II.2.1).

Since the principal $r \times r$ minor
of $\Jac$ is a non-zero polynomial $h \in A$, we can write  a basis of the solution
space of (\ref{system}) over $F$ by choosing the first $r$ variables as leading, and 
the last $s$ variables as free:
\begin{equation*}
\left\{ \frac{q_{1j}}{h} \frac{\del}{\del x_1} + \ldots
+ \frac{q_{rj}}{h} \frac{\del}{\del x_r} + \frac{\del}{\del x_{r+j}} 
\, \bigg| \, j = 1, \ldots s \right\} ,
\end{equation*}
where $q_{ij}\in A$. Multiplying by $h$, we obtain vector fields 
\begin{equation}
\label{base}
\tau_j = q_{1j} \frac{\del}{\del x_1} + \ldots
+ q_{rj} \frac{\del}{\del x_r} + h \frac{\del}{\del x_{r+j}} \in \DD.
\end{equation}
Of course these need not generate $\DD$ as an $A$-module.

We will use local parameters to establish local properties of ideals in $\DD$. Recall that
for a non-singular point $P$ with local parameters $t_1, \ldots, t_s \in A$, there exists a 
unique embedding of $A$ 
\begin{equation}
\label{embedA}
\pi: \ A \hookrightarrow R
\end{equation}
into the algebra of formal power series $R = k[[t_1, \ldots, t_s]]$ 
such that $\pi (t_i) = t_i$, $i = 1, \ldots, s$  (\cite{Sha}, Section II.2.2, Theorems 3, 4, 5).

Let $\m_0$ be the ideal in $R$ of power series without a constant term. Consider 
descending filtrations in $A$ and $R$:
\begin{align}
\label{filter}
\begin{split}
&A \supset \m_P \supset \m_P^2 \supset \m_P^3 \supset \ldots \\
&R \supset \m_0 \supset \m_0^2 \supset \m_0^3 \supset \ldots
\end{split}
\end{align}

We define topologies on $A$ and $R$ by taking (\ref{filter}) to be bases of open neighbourhoods of $0$. Since $\mathop\cap\limits_{j = 1}^\infty \m_P^j = (0)$ 
 (\cite{Sha}, Appendix, Section 6, Proposition 4) and 
$\mathop\cap\limits_{j = 1}^\infty \m_0^j = (0)$ , these topologies are separable.

The argument at the beginning of Section II.2.2 in \cite{Sha} shows that the subspace
$k [t_1, \ldots, t_s]$ is dense in $A$. By construction of $\pi$ (loc. cit.), 
$\pi (\m_P^j) \subset \m_0^j$, hence the map $\pi$ is continuous.

Next consider the Lie algebra $\Der_k\, R$ of derivations of $R$. Since for 
$d \in  \Der_k\, R$ we have $d (\m_0^j) \subset \m_0^{j-1}$, we conclude that
every derivation of $R$ is continuous and hence
\begin{equation*}
d \left( \sum_{u \in \Z_+^s} c_u t^u \right) =  \sum_{u \in \Z_+^s} c_u d(t^u) .
\end{equation*}
It follows that every derivation in $\Der_k\, R$ is determined by its values on 
$t_1, \ldots, t_s$ and $\Der_k\, R \cong \HD$, where
\begin{equation*}
\HD = \mathop\oplus\limits_{j=1}^s R \frac{\del}{\del t_j}.
\end{equation*}
The same argument shows that all derivations of $A$ are continuous as well.

Define a filtration on the Lie algebra $\HD$:
\begin{equation*}
\HD = \HD_{-1} \supset \HD_0 \supset \HD_1  \supset \ldots  \ ,
\end{equation*}
where $\HD_{i-1} = \m_0^i \HD$. Then $[\HD_i , \HD_j] \subset \HD_{i+j}$.
The associated graded Lie algebra 
\begin{equation*}
\grD = \HD_{-1} / \HD_0 \oplus \HD_0 / \HD_1 \oplus \HD_1 / \HD_2 \oplus \ldots
\end{equation*}
is isomorphic to $W_s$. We define the projection map $\omega: \, \HD \rightarrow \grD$, where for $\tau \in \HD_j \backslash \HD_{j+1}$ we set $\omega(\tau) = \tau + \HD_{j+1} \in \HD_j / \HD_{j+1}$.

The projection map $\omega$ is not additive, but it has the property that if
$[\omega(\tau), \omega(\eta)] \neq 0$ then 
$\omega([\tau, \eta]) = [\omega(\tau), \omega(\eta)]$.

\begin{prp}
Let $t_1, \ldots, t_s$ be local parameters at a non-singular point $P \in X$. Then there 
exists a unique embedding
\begin{equation}
\label{embedD}
\whp : \DD \hookrightarrow \HD ,
\end{equation}
such that the following diagram is commutative:
\begin{equation*}
\begin{CD}
\DD \times A @>>> A \\
@V{\whp \times \pi}VV @VV{\pi}V\\
\HD \times R @>>> R
\end{CD}
\end{equation*}
\end{prp}

In the above commutative diagram the horizontal arrows are the action of a Lie algebra
by derivations.

\begin{proof}
Since a derivation in $\HD$ is determined by its values on $t_1, \ldots, t_s$, there is a unique way to define $\whp$:
\begin{equation*}
\whp (\mu) = \sum_{j=1}^s \pi ( \mu(t_j)) \frac{\del}{\del t_j} \ \text{ \ \ for \ } 
\mu \in \DD.
\end{equation*}
It follows that for every polynomial $q \in k[t_1, \ldots, t_s]$,
\begin{equation*}
\whp(\mu) \pi(q) = \pi (\mu(q)) .
\end{equation*} 
Since both $\whp(\mu) \circ \pi$ and $\pi \circ \mu$ are continuos maps 
$A \rightarrow R$, and $k[t_1, \ldots, t_s]$ is dense in $A$, we conclude that
\begin{equation}
\label{comm}
\whp(\mu) \pi(f) = \pi (\mu(f)) 
\end{equation}
for all $f\in A$, which establishes the commutativity of the diagram.

To prove that $\whp$ is a homomorphism of Lie algebras, it is sufficient to show that
$\whp ([\mu, \tau]) t_j = [\whp(\mu), \whp(\tau)] t_j$ for $\mu, \tau \in \DD$, 
$j = 1, \ldots, s$, which follows immediately from (\ref{comm}).


Finally, we need to show that $\Ker \whp = (0)$. Suppose $\mu \in \Ker \whp$.
Then $\mu(t_j) = 0$ for all $j = 1,\ldots, s$. 
However, $\mu$ is continuous on $A$ and
$k[t_1, \ldots, t_s]$ is dense in $A$. Thus $\mu(A) = 0$ and $\mu = 0$. 
\end{proof}

Let us point out that the images of derivations (\ref{base}) in $R$ are 
$\whp(\tau_j) = \pi(h) \frac{\del}{\del t_j}$, $j = 1, \ldots, s$. Note that the power 
series $\pi(h)$ has a non-zero constant term $h(P)$.

\begin{prp}
\label{local}
Let $J$ be a non-zero ideal in $\DD$. Let $P$ be a non-singular point of $X$.
Then there exist $\mu \in J$ and $f \in A$ such that $\mu (f) (P) \neq 0$.  
\end{prp}
\begin{proof}
Let $t_1, \ldots, t_s$ be local parameters (\ref{param}) at $P$. Choose a non-zero 
vector field $\eta \in J$ and consider its image under the embedding 
$\DD \hookrightarrow \mathop\oplus\limits_{j=1}^s k [[t_1, \ldots, t_s]]
\frac{\del}{\del t_j}$, given by the previous proposition. To simplify notations, we
will identify $\whp(\eta)$ and $\eta$, $\pi(h)$ and $h$, etc. 
Let 
\begin{equation*}
\omega(\eta) = \sum_{j=1}^s u_j (t) \frac{\del} {\del t_j} \neq 0 ,
\end{equation*}
where polynomials $u_j (t)$ are homogeneous of the same degree. Pick a non-zero polynomial $u_{j_0}$ and let $t_1^{k_1} \ldots t_s^{k_s}$ be a monomial that occurs
in $u_{j_0}$. Then
\begin{equation*}
\left( \frac{\del}{\del t_1} \right)^{k_1} \ldots \left( \frac{\del}{\del t_s} \right)^{k_s}
u_{j_0} (t)
\end{equation*}
is a non-zero multiple of $1$.
Consider now
\begin{equation*}
\mu = \ad (\tau_1)^{k_1} \ldots \ad (\tau_s)^{k_s} \eta \in J .
\end{equation*}

Since $\omega(\tau_j) = h(P) \frac{\del}{\del t_j}$, it is easy to see that
$\omega(\mu) \in \HD_{-1} / \HD_0$ and contains $\frac{\del}{\del t_{j_0}}$, 
and hence $\mu (t_{j_0}) (P) \neq 0$. 
\end{proof}
Note that in the above calculation we used the fact that the characteristic of the field is $0$.

\begin{cor}
Any non-zero ideal $J$ in $\DD$ is ample at every non-singular point $P$, i.e., vector fields in $J$, specialized at $P$, span the tangent space at $P$. 
\end{cor}
\begin{proof}
We can immediately see that the vector fields $[t_{j_0} \tau_i,  \mu] \in J$, $i = 1, \ldots, s$, specialized at $P$, span the tangent space at that point.
\end{proof}
\begin{cor}
\label{double}
Let $J$ be a non-zero ideal in $\DD$ and let $P$ be a non-singular point of $X$.
Then there exists $\mu \in J$ and $g \in A$ such that $\mu (\mu (g))(P) \neq 0$.
\end{cor}
\begin{proof}
See the proof of Lemma \ref{nonzero}.
\end{proof}

Now we have obtained local information about ideals in $\DD$. In order to establish 
simplicity of $\DD$ we need to globalize the local information.

\begin{thm}
\label{main}
Let $X$ be a smooth irreducible affine variety over an algebraically closed field $k$ of characteristic $0$. Then the Lie algebra $\DD$ of polynomial vector fields on $X$ is simple. 
\end{thm}
\begin{proof}
Let $J$ be a non-zero ideal in $\DD$. For each point $P \in X$ we apply Proposition \ref{local} and its Corollary \ref{double} to fix $\mu \in J$, $f, g \in A$ such that 
$\mu(f) (P) \neq 0$, $\mu (\mu (g)) (P) \neq 0$. Let $I_P$ be the principal ideal
in $A$ generated by $\mu(f) \mu (\mu (g))$. Set  
\begin{equation*}
I_0 = \sum_{P\in X} I_P  \lhd A.
\end{equation*}
By Lemma \ref{grab}, $I_0 \DD \subset J$.
However for every $P \in X$ the ideal $I_0$ contains a polynomial that does not vanish
at $P$. By Hilbert's Nullstellensatz (\cite{Sha}, Section I.2.2), $1 \in I_0$, which implies $J = \DD$. 
\end{proof}

Finally, let us show that the Lie algebra of polynomial vector fields on a variety with singular points, is not simple. In what follows, we assume that $k$ is algebraically closed, but make no restriction on characteristic.

Consider the ideal $\Ising \lhd A$ defining the singular locus on $X$.

\begin{prp}
\label{inv}
For every $\eta \in \DD$, $\eta (\Ising) \subset \Ising$.
\end{prp}
\begin{proof}
Recall that $\Ising$ is generated by $r \times r$ minors $\detab$ of the Jacobian 
matrix
\begin{equation*}
\Jac = 
\left(
\frac{\del}{\del x_1}
\begin{pmatrix}
f_1 \\ f_2 \\ \ldots \\ f_m \\
\end{pmatrix} \ \ 
\frac{\del}{\del x_2}
\begin{pmatrix}
f_1  \\ f_2 \\ \ldots \\ f_m \\
\end{pmatrix} \ \ 
\ldots \ \ 
\frac{\del}{\del x_n}
\begin{pmatrix}
f_1 \\  f_2 \\ \ldots \\ f_m \\
\end{pmatrix}
\right) .
\end{equation*}
Let $\eta = \sum\limits_{j=1}^n g_j \frac{\del}{\del x_j} \in\DD$. 
We need to show that $\eta (\detab) \in \Ising$. 

Applying $\eta$ column-wise, we get
\begin{equation*}
\eta(\detab) = \sum_{i = 1}^r \det
\left(
\frac{\del}{\del x_{\alpha_1}}
\begin{pmatrix}
f_{\beta_1} \\ f_{\beta_2} \\ \ldots \\ f_{\beta_r} \\
\end{pmatrix}\ \ 
\ldots \ \   
\eta \frac{\del}{\del x_{\alpha_i}}
\begin{pmatrix}
f_{\beta_1} \\ f_{\beta_2} \\ \ldots \\ f_{\beta_r} \\
\end{pmatrix} \ \ 
\ldots \ \ 
\frac{\del}{\del x_{\alpha_r}}
\begin{pmatrix}
f_{\beta_1} \\ f_{\beta_2} \\ \ldots \\ f_{\beta_r} \\
\end{pmatrix}
\right) .
\end{equation*}
Since
\begin{equation*}
\eta \frac{\del}{\del x_{\alpha_i}} = \left[\eta, \frac{\del}{\del x_{\alpha_i}}\right]
+ \frac{\del}{\del x_{\alpha_i}} \eta 
= - \sum_{j=1}^n \frac{\del g_j}{\del x_{\alpha_i}} \frac{\del}{\del x_j}
+ \frac{\del}{\del x_{\alpha_i}} \eta 
\end{equation*}
and $\eta (f_1) = \ldots = \eta(f_m) = 0$, we get 
\begin{equation*}
\eta(\detab) = -\sum_{i = 1}^r \sum_{j=1}^n 
\frac{\del g_j}{\del x_{\alpha_i}}
\det
\left(
\frac{\del}{\del x_{\alpha_1}}
\begin{pmatrix}
f_{\beta_1} \\ f_{\beta_2} \\ \ldots \\ f_{\beta_r} \\
\end{pmatrix}
\ldots 
\frac{\del}{\del x_j}
\begin{pmatrix}
f_{\beta_1} \\ f_{\beta_2} \\ \ldots \\ f_{\beta_r} \\
\end{pmatrix}  
\ldots 
\frac{\del}{\del x_{\alpha_r}}
\begin{pmatrix}
f_{\beta_1} \\ f_{\beta_2} \\ \ldots \\ f_{\beta_r} \\
\end{pmatrix}
\right) .
\end{equation*}
and hence $\eta(\detab) \in \Ising$.
\end{proof}

\begin{thm}\label{thm-sing}
Let $X$ be an irreducible affine variety with singular points over an algebraically closed field $k$. 
For $i \in \N$ set
\begin{equation*}
J_i = \left\{ \mu \in \DD \ \big| \ \mu(A) \subset \Ising^i \right\} .
\end{equation*}
Then

(a) $J_i$ is an ideal in $\DD$,

(b) $J_i \neq (0)$, 

(c) $J_i \neq  \DD$ for large enough $i$.
\end{thm}
\begin{proof}
Let $\mu \in J_i$, $\eta \in \DD$. We need to show that $[\mu, \eta] \in J^i$.
\begin{equation*}
[\mu, \eta] (A) \subset \mu \eta (A) + \eta \mu (A) 
\subset \mu (A) + \eta(\Ising^i) 
\subset \Ising^i + \Ising^{i-1} \eta (\Ising)
\subset \Ising .
\end{equation*}
Here in the last step we used Proposition \ref{inv}. This completes the proof of part (a).

To prove part (b), take non-zero $\eta \in \DD$ and a non-zero $f \in \Ising^i$. Then 
$f \eta$ is a non-zero element of $J_i$.

Finally, let $P$ be a singular point of $X$ and let $\m_P \lhd A$ be the ideal of functions vanishing at $P$. Then $\Ising \subset \m_P$ and for all $i$, 
$\Ising^i \subset \m^i_P$. Choose $\tau \in \DD$ and $f \in A$ such that $\tau(f) \neq 0$. Since $\mathop\cap\limits_{i=1}^\infty \m^i_P = (0)$, there exists $i$ such that
$\tau(f) \not\in \m^i_P$. Hence $\tau(f) \not\in \Ising^i$ and $\tau \not\in J_i$. 
\end{proof}

Theorem \ref{main}   and Theorem \ref{thm-sing} imply Theorem \ref{thm-introd}.

\

In the remaining sections we will discuss some examples of Lie algebras of vector fields on affine algebraic varieties. In what follows $k$ will be an algebraically closed field of characteristic $0$.

\section{Structure of $A$ as a module over the Lie algebra of vector fields}
In this section we discuss the structure of the algebra of polynomial functions $A$ on an irreducible smooth affine variety $X$ as a module over the Lie algebra of vector fields $\DD$. The following result shows that 
 $A$ has length $2$ as a module over $\DD$ and describes the irreducible quotients.

\begin{thm}\label{thm-functions}
\label{functions}
Let $A$ be the algebra of polynomial functions on an irreducible smooth affine variety $X$. Let $\DD$ be the Lie algebra of polynomial vector fields on $X$. As a $\DD$-module, $A$ has a unique proper submodule, which is 1-dimensional and consists of the constant functions.
\end{thm}

\begin{proof}
It is sufficient to prove that a $\DD$-submodule $M \subset A$, generated by a non-constant function $g$ coincides with $A$. Since
$$ (f_1 \eta) f_2 = f_1 (\eta f_2)$$
for $f_1, f_2 \in A$, $\eta \in \DD$, we conclude that $\DD$-submodule 
$$N = U(\DD) \DD g \subset M$$
is an ideal in $A$. 
Fix an arbitrary point $P \in X$ and consider the embedding $\pi$ (\ref{embedA}),
of $A$ into the algebra $R$ of power series in the local parameters $t_1,\ldots, t_s$ at $P$, and $\whp$ (\ref{embedD}) of $\DD$ into $\Der R$. Note that $\pi (g)$ is a non-constant power series. Since $\whp(\DD)$ contains derivations with the lowest order terms $\frac{\partial}{\partial t_1}, \ldots, \frac{\partial}{\partial t_s}$, we conclude that
$N$ contains a function that does not vanish at $P$.
Finally, applying the Nullstellensatz, we get that $N = A$.
\end{proof}

\subsection{Lie algebra of vector fields on a sphere}
Next we consider example of the algebra of polynomial functions on an $(N-1)$-dimensional sphere identified with the factor algebra $A  = k[x_1, \ldots, x_N] / \left< x_1^2 + \ldots + x_N^2 - 1\right>$. 
The results of this section are probably well known but we include them for completeness. 

We can use Proposition \ref{threedef} to show that the Lie algebra $\DS$ of polynomial vector fields on  $\Sp^{N-1}$ is generated (as an $A$-module) by the vector fields
$\Delta_{ab} = x_b \frac{\partial} {\partial x_a} -  x_a \frac{\partial} {\partial x_b}$,
$1 \leq a \neq b \leq N$. 
The generators satisfy the following relation:
$$x_c \Delta_{ab} + x_a \Delta_{bc} + x_b \Delta_{ca} = 0.$$

Consider the action $\rho$ of $SL_N (k)$ (or $GL_N (k)$) on $\Sp^{N-1}$: 
\begin{equation}
\label{act}
\rho(g) v = \frac{gv}{\left| gv \right|},
\end{equation}
 for $g \in  SL_N (k)$, $v \in \R^3$ with $\left| v \right| = 1$. Because of the norm expression in the denominator, this action is not polynomial (it does not induce the action of   $SL_N (k)$ on the algebra $A$ of polynomial functions). Nonetheless the corresponding infinitesimal action of the Lie algebra $sl_N (k)$ is polynomial. This yields the following standard embedding of $sl_N (k)$ in the Lie algebra $\DS$ of polynomial vector fields on $\Sp^{N-1}$.

\begin{lem}
\label{embed}
Consider Lie algebra $sl_N (k)$ spanned by matrices $\{ E_{ab}, E_{aa} - E_{bb} \, | \, 
 1 \leq a \neq b \leq N \}$.
The embedding $\tau_N$ of $sl_N (k)$ into the Lie algebra $\DS$ of polynomial vector fields on $\Sp^{N-1}$ is given by
\begin{align}
\label{slna}
E_{ab} \mapsto &\sum_{p=1}^N x_b x_p \Delta_{ap},\\
\label{slnb}
E_{aa} - E_{bb} \mapsto &\sum_{p=1}^N x_a x_p \Delta_{ap} - x_b x_p \Delta_{bp}.
\end{align}

The embedding of the subalgebra $so_N (k)$ of skew-symmetric matrices is given by
$$E_{ab} - E_{ba} \mapsto \Delta_{ab}.$$
\end{lem}

The above action of the group $GL_N (k)$ on $\Sp^{N-1}$ is transitive, so $\Sp^{N-1}$
forms a single orbit. Consider a point $Q = (1, 0, \ldots, 0) \in \Sp^{N-1}$. 
The stabilizer of this point is a maximal parabolic subgroup $P$ with the Levi factor $GL_1 (k) \times GL_{N-1} (k)$.

We would like to exhibit a connection to the representation theory of locally compact groups. We will specialize $k = \R$ for this discussion. Consider a representation of
$GL_N (\R)$ induced from the trivial 1-dimensional representation of the parabolic subgroup $P$:
$$S = \Ind_P^{GL_N (\R)} \R = \left\{ f \in \CC^\infty (GL_N (\R) \, | \, f(gp) = f(g)
\hbox{\rm \ for all \ } g\in GL_N (\R), p \in P \right\} .$$
Representations of this type were studied in \cite{HL}. 
We point out that the module $S$ is isomorphic to the space of $\CC^\infty$-functions on $\Sp^{N-1}$ via the pull-back of the map
$GL_N (\R) \rightarrow \Sp^{N-1}$, $g \mapsto \frac{g Q} {\left| gQ \right|}$.
The algebra $A$ of polynomial functions is a subspace in $S$, which, by Lemma \ref{embed}, is invariant under the action of the Lie algebra $gl_N (\R)$ (the identity matrix acts trivially, so the action reduces to the subalgebra $sl_N (\R)$).

Observe that the Lie subalgebra $so_N (\R)$ acts on $A$ in a locally finite-dimensional way. This follows from the last claim of Lemma \ref{embed}, as the operators $\Delta_{ab}$ do not increase the degree of polynomials, hence every vector in $A$ belongs to a finite-dimensional $so_N (\R)$-invariant subspace. It then follows that 
$A$ is a direct sum of finite-dimensional irreducible $so_N (\R)$-modules and hence the action of $so_N (\R)$ integrates to the action of the group $SO_N (R)$. The latter statement also follows from the fact that the restriction of (\ref{act}) to $SO_N$ is polynomial.

Therefore, the algebra $A$ of polynomial functions on $\Sp^{N-1}$ has a structure of an 
$(sl_N, SO_N)$-module and it is a dense subspace in the induced module $S$.  

Since the defining relation of the sphere
$x_1^2 + \ldots + x_N^2 = 1$ is not homogenous in degree, but preserves the parity of the degree, $\Z-grading$ by degree on the space of polynomials induces $\Z_2$-grading on $A$, $A = A_{ev} \oplus A_{odd}$. We note that each of these subspaces is invariant under the action of $sl_N$. To simplify notations, we will consider only $A_{ev}$. The case of $A_{odd}$ is analogous.

Let us recall the structure of $A$ as a module for $so(N)$ (we refer to \cite{Vi}, Chapter IX for details). Here
we assume that $N \geq 3$, as $N=2$ case is easier, but slightly different. The space $A_{ev}$ has an increasing filtration by subspaces $\T_{2k}$, which are spanned by monomials of even degrees not exceeding $2k$. Each subspace $\T_{2k}$ is invariant under
$so_N$ and can be identified as a space of homogeneous polynomials of degree $2k$. In this interpretation the embedding
$\T_{2k} \hookrightarrow \T_{2k + 2}$ is given via multiplication by $r^2 = x_1^2 + \ldots + x_N^2$. The space $\T_{2k}$ decomposes 
in a direct sum of irreducible $so(N)$-modules as
\begin{equation*}
\T_{2k} = \Ha_{2k} \oplus r^2 \Ha_{2k - 2} \oplus \ldots \oplus r^{2k} \Ha_0,
\end{equation*} 
where $\Ha_\ell$ is the space of harmonic homogenious polynomials of degree $\ell$, that is polynomials satisfying Laplace equation,
\begin{equation*}
\Delta h = \frac{\del^2 h}{\del x_1^2} + \ldots + \frac{\del^2 h}{\del x_N^2} = 0 . 
\end{equation*}
Since $r^2 = 1$ as a function on a sphere, this gives a decomposition of $A_{ev}$ in a direct sum of irreducible $so(N)$-modules
\begin{equation}
\label{deco}
A_{ev} = \mathop\oplus\limits_{k=0}^\infty \Ha_{2k} .
\end{equation}
Projection of $so_N$-modules $\T_{\ell} \rightarrow \Ha_\ell$ is given by the operator
\begin{equation}
\label{proj}
P = \sum\limits_{k=0}^{\left[ \frac{\ell}{2} \right]} \frac{(-1)^k}{2^k k!} \frac{ (N + 2\ell - 2k - 4)!!} {(N + 2\ell - 4)!!}
r^{2k} \Delta^k .
\end{equation}

It turns out that the action of $sl_N$ spreads the components $\Ha_\ell$ in a localized way.

\begin{lem}
$sl_N (\Ha_\ell) \subset \Ha_{\ell+2} \oplus \Ha_\ell \oplus \Ha_{\ell - 2}.$
\end{lem}
\begin{proof}
Let $h \in \Ha_\ell$. Then it follows from (\ref{slna})-(\ref{slnb}) that $g = X h \in \T_{\ell + 2}$ for any $X \in sl_N$. Moreover, it is straightforward to check that when written as a homogeneous polynomial of degree $\ell + 2$, $g$ satisfies the equation
$\Delta^3 g = 0$. It then follows from the projection formula (\ref{proj}) that $g$ projects only to 3 components,
$\Ha_{\ell+2}$, $\Ha_\ell$ and $\Ha_{\ell - 2}$.  
\end{proof}

Using this lemma, we can prove the following

\begin{prp}\label{prop-sphere}
Let $N \geq 3$.
Let $A$ be the algebra of polynomial functions on $\Sp^{N-1}$. As an $sl_N$-module,
$A$ decomposes into a direct sum of two submodules $A = A_{odd}  \oplus A_{ev}$,
where $A_{odd}$ (resp. $A_{ev}$) is a subspace in $A$ spanned by the monomials of odd (resp. even) degree. The space $A_{odd}$ is an irreducible $sl_N$-module, while
$A_{ev}$ has a unique proper submodule which is 1-dimensional and consists of constant   
functions.
\end{prp}
\begin{proof}
Let us give a sketch of a proof. Since in the decomposition (\ref{deco}) all components are pairwise non-isomorphic, a non-zero 
submodule in $A_{ev}$ (resp. $A_{odd}$) will contain one of the subspaces $\Ha_{\ell}$. All we have to show that we can use $sl_N$ action applied to
$\Ha_{\ell}$ with $\ell > 0$, to generate $\Ha_{\ell+2}$ and $\Ha_{\ell-2}$. This can be established using an explicit calculation.
If we take as $h$, for example, a projection of $x_a x_b x_c^{\ell -2}$ to $\Ha_\ell$ (with $a,b,c$ distinct), and set 
$g = E_{ab} h$, then the projections of $g$ to $\Ha_{\ell + 2}$ and for $\ell > 1$ to $\Ha_{\ell -2}$ will be non-zero. Then any component $\Ha_{2k}$ (resp. $\Ha_{2k-1}$) with $k > 0$ generates $A_{ev}$ (resp. $A_{odd}$), and the claim of the proposition follows.     
\end{proof}

\

\section{Vector fields on hyperelliptic curves}

Consider a hyperelliptic curve $\HH$ given by equation $y^2  = 2 h(x)$, where $h$ is a monic polynomial of odd degree $2m+1 \geq 3$. Denote by $I$ the ideal in $k[x,y]$ generated by $y^2 - 2h(x)$. The curve $\HH$ is non-singular if and only if 
\GCD$\left( h(x), h^\prime(x) \right) = 1$.

The algebra of polynomial functions on $\HH$ is 
$A = k[x,y] / \left< y^2 - 2h(x) \right>$. As a vector space,
$$ A \cong k[x] \oplus y k[x].$$

Let $\Dh$ be the Lie algebra of polynomial vector fields on $\HH$ identified with the derivation algebra on $A$,
$$\Dh = \Der_k (A).$$
Lie algebra $\Dh$ may be realized as a sunquotient of $W_2 = \Der k[x,y]$:
$$\Dh\cong \left\{ \mu \in W_2 \, \big| \, \mu (I) \subset I \right\} 
/  \left\{ \mu \in W_2 \, \big| \, \mu (k[x,y]) \subset I \right\} .$$

Let $\mu = r(x,y) \dx + s(x,y) \dy \in W_2$. 
Note that $\mu$ satisfies $\mu (I) \subset I$ if and only if $\mu(y^2 - 2h(x)) \in I$
or equivalently $2y s - h^\prime (x) r = 0$ in $A$.

As a left $A$-module, $\Dh$ is a submodule in $A \dx \oplus A \dy$:
$$\Dh = \left\{ r \dx + s \dy \, \big| \, r,s \in A, \, ys - h^\prime(x) r = 0 \right\}.$$ 

\begin{prp}
\label{hypell}
(a) Let $\HH$ be a non-singular hyperelliptic curve $y^2 = 2h(x)$, 
\GCD$\left( h(x), h^\prime(x) \right) = 1$. Then the Lie algebra of vector fields $\Dh$
is a free $A$-module of rank $1$ generated by $\tau = y \dx + h^\prime(x) \dy$.

(b) If  \GCD$\left( h(x), h^\prime(x) \right) = d(x) \neq 1$ then $A$-module $\Dh$ is not free and has two generators,  $\tau = y \dx + h^\prime(x) \dy$ and 
$\mu = 2 \frac{h(x)}{d(x)} \dx + y \frac{h^\prime (x)}{d(x)} \dy$ with a relation 
$y \tau = d(x) \mu$.
\end{prp}

 We leave the proof for the reader.


Let us define the degree of a monomial in $A = k[x] \oplus y k[x]$ by
$$\deg (x^k) = 2k, \hbox{\hspace{10pt}} \deg (x^k y) = 2k + 2m + 1.$$
This yields an increasing filtration
\begin{equation}
\label{filtr}
A_0 \subset A_1 \subset A_2 \subset \ldots, 
\end{equation}
where $A_s$ is the space spanned by all monomials of degree less or equal to $s$.
Note that $A_s A_k \subset A_{s+k}$. The degree that we introduced does not define a 
$\Z$-grading since the relation $y^2 = 2 h(x)$ is not homogeneous.

For $f \in A$ define the leading term, LT$(f)$, as the term of the highest degree in the  expansion of $f$ in the basis $\{ x^k, x^k y \, | \, k \in \N \}$. 
We define the degree of $f \in A$ as the degree of its (unique) leading monomial.
Then $\deg (fg) = \deg(f) + \deg(g)$. This follows from the fact that 
the degree is compatible with the relation in $A$, $\deg h(x) = 2 \deg (y)$. 

Using the filtration (\ref{filtr}) we can construct the associated graded algebra
$$\gr A = A_0 \oplus A_1/A_0 \oplus A_2 /A_1 \oplus \ldots .$$
It is easy to see that each graded component of $\gr A$ has dimension at most $1$ and 
$$\gr A \cong k[x,y] / \left< y^2 - 2 x^{2m+1} \right> .$$
There is an embedding 
$$\psi : \, k[x,y] / \left< y^2 - 2 x^{2m+1} \right> \hookrightarrow k[t]$$
given by $\psi (x) = 2 t^2, \psi (y) = 2^{m+1} t^{2m+1}$. Hence we can identify $\gr A$ as a subalgebra in $k[t]$ generated by $t^2$ and $t^{2m+1}$.

Taking the leading term may be viewed as a map
$$\hbox{\rm LT} :\, A \rightarrow \gr A .$$
This map is not additive, but preserves the product: LT$(fg) = $LT$(f)$ LT$(g)$, where
the product in the right hand side is taken in $\gr A$.  

Assume now that $\HH$ is smooth, so that $\Dh = A \tau$.
It is easy to see that for any monomial $u \in A$, $u \neq 1$, we have $\tau (u) \neq 0$ and $\deg \tau(u) = deg (u) + 2m -1$. Hence for any non-zero vector field $g \tau \in \Dh$ and
any non-constant $f \in A$ we have
$$ \deg \left( g \tau (f) \right) = \deg (f) + \deg (g) + 2m -1 .$$
We assign degree to non-zero elements of $\Dh$ by
$$\deg (g \tau) = \deg (g) + 2m -1.$$
Let $\Dh^k$ be the subspace in $\Dh$ of elements of degree less or equal to $k$.
The we get an increasing filtration
$$ (0) \subset \Dh^{2m-1} \subset \Dh^{2m} \subset \Dh^{2m+1} \ldots .$$
Note that $\eta \in \Dh^k$ if and only if $\mu (A_s) \subset A_{s+k}$ for all $s$.
This implies that $[\Dh^k, \Dh^r ] \subset \Dh^{k+r}$. Consider the associated graded 
Lie algebra 
$$\gr \Dh = \Dh^{2m-1} \oplus \Dh^{2m}/\Dh^{2m-1} \oplus  \Dh^{2m+1}/\Dh^{2m} \oplus \ldots .$$

The proof of the following Lemma is direct, and we omit it.

\begin{lem}
Lie algebra $\gr \Dh$ acts on $\gr A$ by derivations. We can identify $\gr \Dh$ with
a $\gr A$-submodule in $\Der k[t]$ generated by $t^{2m} \frac{\partial}{\partial t}$.
\end{lem}

In the theory of simple finite-dimensional Lie algebras semisimple and nilpotent elements play an important role. Here we show that simple Lie algebra $\Dh$ has no nilpotent or semisimple elements.

\begin{thm}\label{thm-hyper}
Let $\HH$ be a non-singular hyperelliptic curve over an algebraically closed field $k$ of characteristic zero. Let $\Dh$ be the Lie algebra of polynomial vector fields on $\HH$. 
Let $\eta$ be a non-zero vector field in $\Dh$.
Then

(1) $\Ker \ad (\eta) = k \eta$.

(2) $\eta \not\in \IM \ad (\eta)$.

(3) $\Dh$ has no non-trivial semisimple elements.

(4) $\Dh$ has no non-trivial nilpotent elements.

\end{thm} 
\begin{proof}
To prove (1) we note that $\eta \in \Ker \ad (\eta)$. If $\Ker \ad (\eta)$ has dimension greater than one then there would exist a non-zero $\nu  \in \Ker \ad (\eta)$ with $\deg \nu \neq \deg \eta$. Projecting $\eta$ and $\nu$ into $\gr \Dh$ we would get two 
commuting non-proportional homogeneous elements in $\gr \Dh \subset \Der k[t]$, which is not possible.

Part (2) follows from the fact that the smalest degree of a non-zero element in $\Dh$ is 
$2m - 1$, hence $\deg [\eta, \nu] \geq \deg \eta + 2m - 1 > \deg \eta$.

It follows from part (2) that a relation $[\nu, \eta] = \lambda \eta$ with $\lambda \in k$,
$\lambda \neq 0$, can not hold for non-zero elements in $\Dh$. This proves (3).

If $\ad \eta$ is nilpotent then there exists a non-zero $\nu \in \IM \ad \eta$ such that 
$[\eta, \nu] = 0$. But by part (1), $\nu$ must be a multiple of $\eta$, which would contradict (2).
\end{proof}

\begin{rem}
Olivier Mathieu communicated to us that it is possible to classify the curves with nilpotent and semisimple derivations by considering the corresponding automorphisms.
\end{rem}

\section{Vector fields on linear algebraic groups}

Let $G$ be a linear algebraic group over $k$. Consider the algebra $A$ of polynomial functions on $G$ and the Lie algebra $\DG$ of polynomial vector fields on $G$. Since $G$ is a smooth affine variety, the Lie algebra $\DG$ is simple.

 The algebra $A$ of functions on $G$ is a Hopf algebra with a comultiplication
\break
$\Delta: \, A \rightarrow A \otimes A$. We shall write
$$\Delta(f) = \sum_i f_i^{(1)} \otimes f_i^{(2)}, \ \text{\ for \  } f\in A.$$
The comultiplication is coassociative: $(\Delta \otimes \id) \circ \Delta = (\id \otimes \Delta) \circ \Delta$, yielding a map $A \rightarrow A \otimes A \otimes A$. We shall write the image of
$f \in A$ under this map as $\sum_i f_i^{[1]} \otimes f_i^{[2]} \otimes f_i^{[3]}$.

 The group $G$ acts on $A$ by left and right shifts:
$$ (L_x f) (y) = f(x^{-1}y),$$
$$ (R_x f) (y) = f(yx), \text{\ for \  } f\in A, \ x,y \in G.$$

 Left and right shifts may be expressed via coproduct (see \cite{Hu}, Section 9.3):
\begin{equation}
\label{leftshift}
L_{x^{-1}} f = \sum_i f_i^{(1)} (x) f_i^{(2)},
\end{equation}
\begin{equation}
\label{rightshift} 
R_{x} f = \sum_i f_i^{(2)} (x) f_i^{(1)}.
\end{equation}

 The group $G$ acts on $\DG$ by conjugation with left (resp. right) shifts:
$$C^L: \, G \times \DG \rightarrow \DG,$$
where $C^L_x \eta = L_{x} \circ \eta \circ L_{x^{-1}}$
(resp. $C^R_x \eta = R_{x} \circ \eta \circ R_{x^{-1}}$).

 Lie algebra $\DG$ has two finite-dimensional subalgebras, $\LL^L$ and $\LL^R$, of left- and 
right-invariant vector fields. The Lie algebra $\LL^L$ (resp. $\LL^R$) consists of those vector fields in $\DG$ that commute with the left shifts $L_x$ (resp. right shifts $R_x$) for all 
$x \in G$.  

 The Lie algebras $\LL^L$ and $\LL^R$ are finite-dimensional. Each of them can be identified with the tangent space to $G$ at identity element $e \in G$.

 Let $\m_e \triangleleft A$ be the maximal ideal of functions vanishing at $e$. The tangent space to $G$ at $e$ can be defined as $\left( \me / \me^2 \right)^*$. For a tangent vector 
$\varphi \in \Te$ we define the operation of the directional derivative $\widehat \varphi:
\, A \rightarrow k$, where $\widehat{\varphi} (f) = \varphi \left( f - f(e) 1 \right)$.
This directional derivative satisfies $\wv (fg) = g(e) \wv(f) + f(e) \wv(g)$.

 We have the isomorphisms (see Section 9.2 in \cite{Hu}):  
$$ \theta_L : \, \Te \rightarrow \LL^L, $$
$$ \theta_R : \, \Te \rightarrow \LL^R, $$
where for $\varphi \in \Te$, $f \in A$ and $x \in G$,
$$\left( \theta_L (\varphi) f \right) (x) = \wv \left( L_{x^{-1}} f \right),$$
$$\left( \theta_R (\varphi) f \right) (x) = \wv \left( R_x f \right).$$

 It is easy to check that $\theta_L (\varphi)$ (resp. $\theta_R (\varphi)$) is a left- (resp. right- ) invariant vector field.

 Note that the sublagebras $\LL^L, \LL^R \subset \DG$ commute with each other. Indeed, using
(\ref{leftshift})-(\ref{rightshift}), as well as coassociativity, we get for $\varphi, \psi \in \Te$:
\begin{equation}
\theta_L (\varphi) \theta_R (\psi) f = \theta_R (\psi) \theta_L (\varphi) f =
\sum_i \widehat{\psi} (f_i^{[1]}) \widehat{\varphi} (f_i^{[3]}) f_i^{[2]}.
\end{equation}   

Lie groups are parallelizable manifolds. For linear algebraic groups, an algebraic counterpart of this geometric statement is the following result which can be found 
in \cite{KMRT}, Proposition 21.3. We include the proof for convenience of the reader.

\begin{thm}\label{thm-alggr}
\label{freeleft}
Let $G$ be a linear algebraic group. Then the Lie algebra $\DG$ of polynomial vector fields on 
$G$ is a free $A$-module, generated by $\LL^L$ (resp. $\LL^R$):
$$ \DG \cong A \otimes \LL^L \cong A \otimes \LL^R .$$
\end{thm}
\begin{proof}
The map $\epsilon: \, A \otimes \LL^L \rightarrow \DG$ is just a multiplication of a vector field by a function.
To establish an isomorphism, we need to construct its inverse $\DG \rightarrow A \otimes \LL^L$.

First, let us define the map
$$\Gamma: \, \DG \otimes \Tf \rightarrow A,$$
$$ \Gamma (\eta \otimes f) (x) = \left( (C^L_{x^{-1}} \eta) f \right) (e) = \eta (L_x f) (x),$$
where $\eta\in\DG$, $f \in \m_e$, $x \in G$. It is easy to see that this map is well-defined,
if $f = f_1 f_2 \in \m_e^2$ then $(C^L_{x^{-1}} \eta) (f_1 f_2) = f_2 (C^L_{x^{-1}} \eta) (f_1) +
f_1 (C^L_{x^{-1}} \eta) (f_2)$, which evaluates at $e$ to $0$.

 Fix a basis $\{ f_1, \ldots, f_n \}$ of $\Tf$ and a dual basis $\{ \varphi_1, \ldots , \varphi_n \}$ of $\Te$. Then for $f \in \m_e$,
$$ f = \sum_{i=1}^n \varphi_i (f) f_i \ \ \mod \m_e^2 .$$

We construct the map $\delta: \, \DG \rightarrow A \otimes \LL^L$ as
$$\delta(\eta) = \sum_{i=1}^n \Gamma ( \eta \otimes f_i) \otimes \theta_L (\varphi_i).$$

Let us show that $\delta$ is the inverse map to $\epsilon$. One easily checks that
 $\epsilon(\delta (\eta)) = \eta$.
  Since both $\epsilon$ and $\delta$ are $A$-linear, to show that $\delta \circ \epsilon = \id$, it is sufficient to check that $\delta (\epsilon ( 1 \otimes \theta_L (\varphi_j) )) = 
\theta_L (\varphi_j)$ . 
This is equivalent to showing that $\Gamma (\theta_L (\varphi_j) \otimes f_i )$ equals 1 if $i=j$ and 0 otherwise. We have
$$\Gamma (\theta_L (\varphi_j) \otimes f_i ) (x) = \theta_L (\varphi_j) (L_x f_i) (x)$$
$$= \varphi_j \left( L_{x^{-1}} L_x f_i - f_i (e) 1 \right) = \varphi_j (f_i),$$
which yields the desired result.
\end{proof}

We point out that for reductive groups the structure of $\DG$ as an $\LL^L$-module with respect to the adjoint action, can be obtained from Theorem \ref{freeleft} together with the Peter-Weyl theorem:

\begin{thm} (\cite{GW}, Theorem 4.2.7). Let $G$ be a reductive linear algebraic group. Let $\widehat{G}$ denote the set of equivalence classes of finite-dimensional polynomial representations of $G$. For $\lambda \in \widehat{G}$ denote by $V_\lambda$ the corresponding $G$-module.
 Then the algebra $A$ of polynomial functions on $G$, viewed as a $G \times G$-module (with respect to left and right action) decomposes in the following way:
$$A = \mathop\oplus\limits_{\lambda \in \widehat{G}} V_\lambda \otimes V_\lambda^* .$$ 
\end{thm}

\end{document}